\documentclass[11pt,reqno]{amsart}

\usepackage{amssymb}
\usepackage{amsmath}
\usepackage{amsthm}
\usepackage{latexsym}
\usepackage{amsfonts}
\usepackage{mathrsfs}
\usepackage{bbm}
\usepackage{titlesec}
\usepackage{color}
\usepackage{hyperref}

\newtheorem{thm}{Theorem}[section]
\newtheorem{cor}[thm]{Corollary}
\newtheorem{lem}[thm]{Lemma}
\newtheorem{prop}[thm]{Proposition}

\theoremstyle{definition}

\newtheorem{rem}[thm]{Remark}
\newtheorem{ex}[thm]{Example}

\titleformat{\section}{\normalfont\bfseries\centering}{\thesection.}{.25em}{}
\titleformat{\subsection}{\normalfont\bfseries}{\thesubsection.}{.25em}{}
\titleformat{\subsubsection}{\normalfont\bfseries}{\thesubsubsection.}{.25em}{}
\titlespacing{\section}{0pt}{*3}{*1.5}
\titlespacing{\subsection}{0pt}{*4}{*0.5}
\titlespacing{\subsubsection}{0pt}{*4}{*0.5}
\numberwithin{equation}{section}
\allowdisplaybreaks

\newcommand{\R}{\ensuremath{\mathbb R}}    
\newcommand{\C}{\ensuremath{\mathbb C}}    
\newcommand{\N}{\ensuremath{\mathbb N}}    
\newcommand{\Z}{\ensuremath{\mathbb Z}}    
\newcommand{\D}{\ensuremath{\mathbb D}}    
\newcommand{\T}{\ensuremath{\mathbb T}}    

\newcommand{\<}{\langle}
\renewcommand{\>}{\rangle}

         \newcommand{\frakB}{\mathfrak B}
         
\newcommand{\calD}{\mathcal D}         
         
\newcommand{\calF}{\mathcal F}         
         
\newcommand{\calH}{\mathcal H}         
         \newcommand{\frakI}{\mathfrak I}

\newcommand{\calL}{\mathcal L}

\newcommand{\calV}{\mathcal V}

\newcommand{\la}{\lambda}

\newcommand{\vphi}{\varphi}

\newcommand{\linspan}{\operatorname{span}}
\renewcommand{\ker}{\operatorname{ker}}
\newcommand{\ran}{\operatorname{ran}}
\newcommand{\dom}{\operatorname{dom}}

\newcommand{\Sra}{\Rightarrow}

\newcommand{\Llra}{\Longleftrightarrow}

\newcommand{\upto}{\uparrow}

\newcommand{\ol}{\overline}

\newcommand{\wt}{\widetilde}
\newcommand{\wh}{\widehat}

\def\bei{\begin{itemize}}

	\def\eni{\end{itemize}}

\newcommand{\ltr}{ L^2(\mathbb R) }

\definecolor{darkgreen}{rgb}{0,0.6,0.1}

\newcommand{\rmref}[1]{{\rm\ref{#1}}}
\newcommand{\braces}[1]{{\rm (}#1{\rm )}}
\newcommand{\Id}{\operatorname{Id}}

\begin{document}
\title[]{Frame Properties of Operator Orbits}

\author[O. Christensen]{Ole Christensen}
\address{{\bf O.~Christensen:} Technical University of Denmark, DTU Compute, Building 303, 2800 Lyngby, Denmark}
\email{ochr@dtu.dk}
\urladdr{http://www.dtu.dk/english/service/phonebook/person?id=4715\&cpid=3712\&tab=1}

\author[M. Hasannasab]{Marzieh Hasannasab}
\address{{\bf M.~Hasannasab:} Technical University of Denmark, DTU Compute, Building 303, 2800 Lyngby, Denmark}
\email{mhas@dtu.dk}
\urladdr{http://www.dtu.dk/english/service/phonebook/person?id=110619\&tab=2\&qt=dtupublicationquery}

\author[F. Philipp]{Friedrich Philipp}
\address{{\bf F.~Philipp:} KU Eichst\"att-Ingolstadt, Mathematisch-Geographische Fakult\"at, Ostenstra\ss e 26, Kollegiengeb\"aude I Bau B, 85072 Eichst\"att, Germany}
\email{fmphilipp@gmail.com}
\urladdr{http://www.ku.de/?fmphilipp}

\begin{abstract}
We consider sequences in a Hilbert space $\calH$ of the form $(T^nf_0)_{n\in I},$ with a linear operator $T$, the index set being either $I = \N$ or $I = \Z$, a vector $f_0\in
\calH$, and answer the following two related questions: (a) {\it Which frames for $\calH$ are of this form with an at least closable operator $T$?} and (b) {\it For which bounded operators $T$ and vectors $f_0$ is $(T^nf_0)_{n\in I}$ a frame for $\calH$?} As a consequence of our results, it turns out that an overcomplete Gabor or wavelet frame can never be written in the form $(T^nf_0)_{n\in \N}$ with a bounded operator $T$. The corresponding problem for $I = \Z$ remains open. Despite the negative result for Gabor and wavelet frames, the results demonstrate that the class of frames that can be represented in the form $(T^nf_0)_{n\in \N}$ with a bounded operator $T$ is significantly larger than what could be expected from the examples known so far.
\end{abstract}

\subjclass[2010]{94A20, 42C15, 30J05}
\keywords{Frames, operator orbits, Gabor frames, contractions}

\maketitle
\thispagestyle{empty}

\section{Introduction}
This paper is a contribution to dynamical sampling, a recent research topic introduced in \cite{acmt}. Dynamical sampling has already attracted considerable attention \cite{AK,A2,ap,cmpp,cmpp2,ch0,ch1,ch2,ch3,chr,p}. In short form, and to be made more precise soon, the question is to analyze the frame properties of sequences in a Hilbert space $\calH$ having the form $(T^n f_0)_{n=0}^\infty,$ where $T: \calH \to \calH$ is a linear operator and $f_0 \in \calH$. In operator theory, the set $(T^n f_0)_{n=0}^\infty$ is called the {\it orbit} of the vector $f$ under the operator $T$.

The main purpose of this paper is to characterize and compare the frame properties of orbits $(T^n f_0)_{n=0}^\infty$ and the bi-infinite orbits $(T^n f_0)_{n=-\infty}^\infty$ with a bounded operator $T$. The first frame characterization of orbits of the type $(T^nf_0)_{n=0}^\infty$ with a {\em normal} operator $T$ appeared in \cite{acmt}. Its necessary and sufficient conditions are very explicit in the sense that they allow for actually checking whether a given orbit with a normal operator is a frame or not. In addition, the characterization is suitable for constructing frame orbits with normal operators as generators. On the other hand, it is very restrictive since ``most'' frame orbits are generated by non-normal operators (see \ Remark \ref{r:comp_normal} for a more detailed description). Hence, it is desirable to find characterizations for general bounded operators. In fact, such a characterization has been found in \cite{cmpp} in the more general framework of ``extended'' operator orbits $(T^nf_j)_{n\in\N,\,j\in J}$, where $J$ is a countable index set. However, the characterization in \cite{cmpp} is neither explicit in the above sense, nor does it allow for a parametrization of all frame orbits in terms of well known objects. Here, we provide characterizations which satisfy the latter desideratum. In particular, while the property of a sequence of vectors being a frame is independent of the chosen indexing, our results show that the answer to the question whether an orbit sequence $(T^n f_0)_{n=0}^\infty$ is a frame or not turns out to be considerably different from the case of the bi-infinite orbits. For example, we will see in Theorem \ref{71218d} that if $(T^n f_0)_{n=0}^\infty$ is  a frame, then $T^n f_0 \to 0$ as $n\to \infty;$ and if $(T^n f_0)_{n=-\infty}^\infty$ is a frame then the vectors $T^n f_0, n\in \Z,$ are norm-bounded from below. This proves in particular that regardless of the chosen ordering, a Gabor frame cannot be represented in the form  $(T^n f_0)_{n=0}^\infty$ with a bounded operator $T.$ As we will see, this no-go result is due to the fact that $(T^n f_0)_{n=0}^\infty$ is the full orbit of $f_0$ under $T$: as a matter of fact, a Gabor frame is always contained in an orbit of a certain bounded operator.

Another objective of this paper is to characterize the frames that can be represented in the form $(T^nf_0)_{n\in I}$, where either $I=\N$ or $I=\Z$ and the linear operator $T$
is closable. The corresponding question for $T$ being bounded was solved in \cite{ch3,chr}. Surprisingly, we will show that
even though closability is a much weaker condition than boundedness, the  conditions on the frame remain the same in these two cases.

The paper is organized as follows. In the rest of this Introduction we state certain
key definitions and results from the literature and present the main new results.
Additional results and the proofs are given in the next four sections. Section \ref{71219a} contains the above-mentioned result for closable operators which is valid for both types of orbits. Section \ref{71219b} presents results that are particular for the classical orbits indexed by $\N$, and Section \ref{71219c} discusses bi-infinite orbits indexed by $\Z$. In Section \ref{80104a} we provide explicit frame constructions $(T^nf_0)_{n\in \N}$ for
operators $T$ that are similar to normal operators. Comparing the result with the
characterization obtained for general bounded operators proves that the class of
frames of the form $(T^nf_0)_{n\in \N}$ with a bounded operator $T$ is much larger than the previously known examples indicate. We also provide a perturbation result which can be used to construct a non-normal frame orbit from a normal one.

In order to make the paper accessible to readers from different communities, an appendix collects the necessary background information on contractions and subspaces of $L^2(\T)$.

\subsection{Definitions and notation} \label{71218b}
Throughout this paper, let $\N = \{0,1,2,\ldots\}.$ We let
$\calH$ denote a complex separable infinite-dimensional Hilbert space. Given Hilbert spaces $\calH_1$ and $\calH_2,$ we let $L(\calH_1,\calH_2)$  denote the set of all bounded operators mapping $\calH_1$ into $\calH_2$. Moreover, $GL(\calH_1,\calH_2)$ will denote the set of all bijective operators in $L(\calH_1,\calH_2)$. As usual, we set $L(\calH) := L(\calH,\calH)$ and $GL(\calH) := GL(\calH,\calH)$.

In order to formulate our results about orbits $(T^nf_0)_{n\in I}$  in an efficient way, we introduce a natural similarity relation between operator-vector pairs $(T,f)$. Considering two Hilbert spaces $\calH_1$ and $\calH_2$ as well as $T_j\in L(\calH_j)$ and $f_j\in\calH_j$, $j=1,2,$ we say that the pairs $(T_1,f_1)$ and $(T_2,f_2)$ are {\em equivalent} (or {\em similar}) {\it via} $V\in GL(\calH_1,\calH_2)$ if
\begin{equation}\label{e:similar}
T_2 = VT_1V^{-1}\qquad\text{and}\qquad f_2 = Vf_1.
\end{equation}
In this case, we write $(T_1,f_1)\cong (T_2,f_2)$.

Let $\T$ denote the unit circle. If $\sigma\subset\T$ is a Borel set, by $\frakB(\sigma)$ we denote the set of all Borel sets which are contained in $\sigma$. For $\sigma\in\frakB(\T)$ we denote by $|\sigma|$ the normalized arc length measure of $\sigma$. In particular, $|\T|=1$. Let $\sigma\in\frakB(\T)$. By $M_\sigma$ we denote the operator of multiplication with the free variable in $L^2(\sigma)$. That is, $M_\sigma : L^2(\sigma)\to L^2(\sigma)$ and $(M_\sigma h)(z) := zh(z)$, $h\in L^2(\sigma)$, $z\in\sigma$. The operator $M_\sigma$ is obviously unitary. By $1_\sigma$ we denote the constant function with value $1$ in $L^2(\sigma)$.

In what follows, we write $L^2 := L^2(\T)$. By $L^2_+$ we denote the subspace of $L^2$ consisting of all $f\in L^2$ with vanishing negative Fourier coefficients, i.e.,
\begin{equation}\label{e:L2+}
L^2_+ := \left\{f\in L^2 : \int_\T z^nf(z)\,d|z| = 0\;\text{ for }n\ge 1\right\},
\end{equation}
where $d|z|$ indicates integration with respect to the normalized arc length measure. We also define $L^2_- := L^2\ominus L^2_+$ and denote the orthogonal projections onto $L^2_+$ and $L^2_-$ by $P_+$ and $P_-$, respectively. On $L^2_+$ define the operator of multiplication with the free variable by
$$
M_\T^+ : L^2_+\to L^2_+,\qquad (M_\T^+ f)(z) = zf(z),\quad f\in L^2_+,\,z\in\T.
$$
By $\frakI$ we denote the class of functions $h\in L^2_+$ such that $|h(z)|=1$ for a.e.\ $z\in\T$. The functions in $\frakI$ are exactly the radial limits of the inner functions on the open unit disk (cf.\ Appendix \ref{ap:contractions}). Therefore, we also call them {\em inner}. Also, let $\frakI^*$ denote the set of all inner functions which are not finite Blaschke products. As a convention, we identify inner functions whose quotient is a (unimodular) constant.

\subsection{An overview of the new results} \label{71218c}
In this section we will collect some of our main new contributions and relate them
to known results and open problems from the literature. We first state a characterization of frames of the forms $(T^nf_0)_{n\in\N}$ and $(T^nf_0)_{n\in\Z}$, respectively, which is a combination of Theorem \ref{t:charac_N} and Theorem \ref{t:charac}.
Given a closed subspace $M$ of the underlying Hilbert space $\calH$, let $P_M$  denote the orthogonal projection onto $M$.
Now, for $h\in\frakI\cup\{0\}$, let
\begin{equation}\label{e:objects}
\calH_h := L^2_+\ominus hL^2_+,\qquad A_h := P_{\calH_h}\left(M_\T^+|_{\calH_h}\right),\qquad\text{and}\qquad\phi_h := P_{\calH_h}1_\T.
\end{equation}
Note that the spaces $\calH_h$, $h\in\frakI^*\cup\{0\}$, are exactly the infinite-dimensional orthogonal complements of the invariant subspaces of the multiplication operator $M_\T^+$ (cf.\ Appendix \ref{ap:subspaces}). The operators $A_h$ are the compressions of $M_\T^+$ to these orthogonal complements.

\begin{thm} \label{71812a} \ \, \
\bei
\item[(a)]
Let $T\in L(\calH)$ and $f_0\in\calH$. Then the following statements are equivalent:
\begin{enumerate}
\item[{\rm (i)}]  The system $(T^nf_0)_{n\in\N}$ is a frame for $\calH$.
\item[{\rm (ii)}] $(T,f_0)\cong (A_h,\phi_h)$ for some $h\in\frakI^*\cup\{0\}$.
\end{enumerate}

\item[(b)]
Let $T\in GL(\calH)$ and $f_0\in\calH$. Then the following statements are equivalent:
\begin{enumerate}
\item[{\rm (i)}]  The system $(T^nf_0)_{n\in\Z}$ is a frame for $\calH$.
\item[{\rm (ii)}] $(T,f_0)\cong (M_{\sigma},1_\sigma)$ for some $\sigma\in\frakB(\T)$.
\end{enumerate}
\eni
\end{thm}

So, while operators generating frame orbits indexed over $\Z$ are similar to multiplication operators, the ones generating frame orbits indexed over $\N$ are similar to special compressions of the multiplication operator $M_\T^+$. Note that the compressions $A_h$ play an important role in Nikolskii's book \cite{n}. They also serve as a prominent example for the so-called $C_0$-contractions (cf.\ Appendix \ref{ap:contractions} and \cite[Ch.\ III, Prop.\ 4.3]{nfbk}). Their spectrum in the open unit disk consists of isolated eigenvalues which are exactly the zeros of $h$.

An important consequence of Theorem \ref{71812a} is that the class of frames that can be represented in the form  $(T^nf_0)_{n\in \N}$ with a bounded operator $T$ is much larger than what could be expected from the examples known so far. We will discuss this in more detail in Remark \ref{r:comp_normal}.

Assuming that the orbit $(T^nf_0)_{n\in\N}$ is a frame for $\calH$
for some fixed $f_0\in \calH,$ it is natural to ask for a characterization
 of {\em all} $f\in \calH$ such that $(T^nf)_{n\in\N}$ is a frame for $\calH$. For
 $T\in L(\calH)$, we indeed characterize these vectors in Proposition \ref{p:VT} (see also Remark \ref{r:explicit} for a more explicit representation). Analogue results for bi-infinite orbits are given in Section \ref{71219c}.

The next result combines Corollary \ref{c:strst} and Corollary \ref{c:notendstozero}. The conclusion about $T^*$ in \eqref{71219d} is already known \cite{ap} and is only included here for completeness and direct comparison between (a) and (b). Recall that a frame is called {\em overcomplete} if it is not a Riesz basis.

\begin{thm} \label{71218d} \ \, \
\bei

\item[(a)]
Let $T\in L(\calH)$ and $f_0\in\calH$ such that $(T^nf_0)_{n\in\N}$ is an overcomplete frame for $\calH$. Then for each $f\in\calH$ we have that
\begin{eqnarray} \label{71219d}
T^nf\to 0
\qquad\text{and}\qquad
(T^*)^{n}f\to 0\quad\text{as }n\to\infty.
\end{eqnarray}
\item[(b)] Let $T\in GL(\calH)$ and $f_0\in\calH$, such that $(T^nf_0)_{n\in\Z}$ is a frame for $\calH$ with frame bounds $A$ and $B$. Then for all $f\in\calH$ and all $n\in\Z$ we have
$$
\|T^nf\|\,\ge\,\sqrt{\frac A B}\,\|f\|\qquad\text{and}\qquad\|(T^*)^{n} f\|\,\ge\,\sqrt{\frac A B}\,\|f\|.
$$
\eni

\end{thm}

Theorem \ref{71218d} (a) answers a question posed in \cite{ch0}. In particular,
the result shows that the classical overcomplete frames in $\ltr$, i.e., frames of translates, Gabor
frames, and wavelet frames cannot be represented in the form $(T^nf_0)_{n\in\N}$ with a bounded operator $T: \ltr \to \ltr;$ indeed, all these frames consist of
elements of equal norm, which contradicts \eqref{71219d}.
At present it is still an open problem whether there exist
overcomplete Gabor frames or wavelet frames that can be represented in the form $(T^nf_0)_{n\in\Z}$ with
a bounded operator $T: \ltr \to \ltr.$ Note that there exist  overcomplete
frames of translates for subspaces of $\ltr$ with such a representation: indeed, considering
the translation operator $T_{1/2}f(x)= f(x-1/2),$ the sequence
$(\mbox{sinc}(\cdot - n/2))_{n\in \Z}=  (T_{1/2}^n\mbox{sinc})_{n\in \Z} $ is an
overcomplete frame for the Paley-Wiener space because $(\mbox{sinc}(\cdot - n))_{n\in \Z}$
is a basis for that space.

It is interesting to notice that even though an overcomplete Gabor frame cannot be represented in the form $(T^nf_0)_{n\in\N}$ with a bounded operator $T$, any Gabor frame can actually be written as a subsequence of such a sequence. Indeed, a fundamental result by Halperin, Kitai, and Rosenthal \cite{kitai} states that every linearly independent countable sequence in a Hilbert space $\calH$ is contained in the orbit of a certain bounded operator $T: \calH \to \calH.$ Thus, we have the following result, which in particular applies to every Gabor system with a nonzero window:

\begin{prop}[{\cite{kitai}}]
Consider any linearly independent frame $(f_n)_{n\in \N}$ in a Hilbert space $\calH$. Then there exists a bounded operator $T: \calH \to \calH$
and a sequence $(\alpha_n)_{n\in \N}\subset \N$ such that
$$
(f_n)_{n\in \N} = (T^{\alpha_n}f_1)_{n\in \N}.
$$
\end{prop}

Theorem \ref{71218d} (a) also shows that a frame of the form $(T^nf_0)_{n\in\N}$
cannot contain an infinite Riesz sequence. The mentioned example of the
frame $(\mbox{sinc}(\cdot - n/2))_{n\in \Z}$ in the
Paley-Wiener space shows that this restriction does not apply to frames
of the form $(T^nf_0)_{n\in\Z}$.

\section{Frame orbits generated by closable operators}\label{71219a}
Throughout this section we will consider sequences $F = (f_n)_{n\in I}$ in a Hilbert space $\calH$ with indexing over $I=\N$ or $I=\Z$. A natural question to ask is whether there exists a linear operator $T$ such that $f_{n+1} = Tf_n$ for all $n\in I$? In the affirmative case, the operator $T$ obviously maps the subspace $\calD_F := \linspan\{f_n : n\in I\}$ into itself; furthermore the operator is surjective if $I = \Z$. In \cite{ch3,chr} it was proved that such an operator $T : \calD_F\to\calD_F$ exists if and only if the set $\{f_n : n\in I\}$ is linearly independent. The operator $T$ then is unique (and bijective if $I=\Z$) and $f_n = T^nf_0$ for all $n\in I$. We call $T$ the {\em generating operator} of $F$. In the case where $F$ is a frame for $\calH$ it was shown in \cite{ch3,chr} that its generating operator is bounded if and only if the kernel of the synthesis operator of $F$ is invariant under the right shift operator on $\ell^2(I)$.

Since invariance of a subspace of $\ell^2(I)$ under the right shift is a fairly strong property, it seems that requiring boundedness of $T$ is very restrictive and relaxing this requirement might lead to a weaker condition. In what follows, we will characterize those frames $F = (f_n)_{n\in I}$ for $\calH$ which can be represented in the form $(T^nf_0)_{n\in I}$, where $I=\N$ or $I=\Z$ and $T$ is a {\em closable} operator, defined on $\calD_F := \linspan\{f_n : n\in I\}$. Recall that an operator $T : \calD_F\to\calH$ is called closable if the closure of its graph in $\calH\times\calH$ is again the graph of an operator (which is then called the closure of $T$, denoted by $\ol T$).

Let $F = (f_n)_{n\in I}$ be a linearly independent frame for $\calH$. Since the frame is complete in $\calH$, its generating operator $T : \calD_F\to\calH$ is obviously densely defined. Therefore, its adjoint $T^*$ exists. Recall that
$$
\dom T^* = \big\{g\in\calH\,|\,\exists h\in\calH\,\forall f\in\calD_F : \<Tf,g\> = \<f,h\>\big\}.
$$
In the sequel, we denote by $L$ and $R$ the left shift and the right shift operator on $\ell^2(I)$, respectively.

\begin{lem}\label{l:DT*}
Let $F = (f_n)_{n\in I}=(T^nf_0)_{n\in I}$ be a linearly independent frame for $\calH$ with generating operator $T : \calD_F\to\calH$. Then $\dom T^*$ is closed and
\begin{align*}
\dom T^*
= \big\{g\in\calH\,|\,\exists h\in\calH\,\forall n\in I : \<f_{n+1},g\> = \<f_n,h\>\big\} = (LC)^{-1}\ran C,
\end{align*}
where $C$ denotes the analysis operator of $F$ and $(LC)^{-1}\ran C$ denotes the
preimage of the set $\ran C$ under the map $LC.$
\end{lem}
\begin{proof}
We have $Tf_n = f_{n+1}$, $n\in I$.  From this and the definition of $\dom T^*$ it is clear that $\dom T^*$ is contained in the right hand side of the first equation. Now, let $g,h\in\calH$ such that for all $n\in I$ we have $\<f_{n+1},g\> = \<f_n,h\>$. Let $f\in\calD_F$, $f = \sum_{n\in I'}c_nf_n$, where $I'\subset I$ is finite and $c_n\in\C$, $n\in I'$. Then
$$
\<Tf,g\> = \sum_{n\in I'}c_n\<f_{n+1},g\> = \sum_{n\in I'}c_n\<f_{n},h\> = \<f,h\>.
$$
This proves that $g\in\dom T^*$. For the second equation we note that for $g,h\in\calH$ we have
$$
\left(\<g,f_{n+1}\>\right)_{n\in I} = LCg\qquad\text{and}\qquad\left(\<h,f_n\>\right)_{n\in I} = Ch.
$$
Thus, $g\in\dom T^*$ if and only if there exists $h\in\calH$ such that $LCg = Ch$, that is, $LCg\in\ran C$. Hence, $\dom T^* = (LC)^{-1}\ran C$. As the pre-image of a closed set under a continuous map, $\dom T^*$ is closed.
\end{proof}

The closed graph theorem immediately yields the following corollary.
\begin{cor}
	Let $F = (f_n)_{n\in I}=(T^nf_0)_{n\in I}$ be a linearly independent frame for $\calH$ with generating operator $T : \calD_F\to\calH.$ Then $T^*$ is bounded.
\end{cor}

The following theorem is our main result in this section. It shows in particular the surprising fact that the generating operator of a linearly independent frame is closable if and only if it is bounded. Note that we do not assume a priori that the given frame is linearly independent. For $I=\N,$ the equivalence of (i) and (iii) was already proved in \cite{chr}.

%
%

\begin{thm}\label{t:boundedly_generated}
	Let $F = (f_n)_{n\in I}$ be a frame for $\calH$ with synthesis operator $U$. Then the following statements are equivalent.
	\begin{enumerate}
		\item[{\rm (i)}]   $F$ has the form $(T^n f_0)_{n\in I}$ with a bounded operator $T$.
		\item[{\rm (ii)}]  $F$ has the form $(T^n f_0)_{n\in I}$ with a closable operator $T$.
		\item[{\rm (iii)}] $\ker U$ is $R$-invariant.
	\end{enumerate}
	In case these equivalent conditions are satisfied, the closure of the generating operator $T$ of $F$ is given by
	\begin{equation}\label{e:formula}
	\ol T = URU^\dagger,
	\end{equation}
	where $U^\dagger = U^*(UU^*)^{-1}$ denotes the pseudo-inverse of $U$.
\end{thm}
\begin{proof}
	Let $C$ denote the analysis operator of $F$, i.e., $C = U^*$. The implication (i)$\Sra$(ii) is trivial. For the other implications we will make use of the easily proved relation
	\begin{equation}\label{e:applied_twice}
	\left[(LC)^{-1}\ran C\right]^\perp = \ol{(LC)^*\ker U} = \ol{UR\ker U}.
	\end{equation}
	(ii)$\Sra$(iii). By Lemma \ref{l:DT*}, $\dom T^* = (LC)^{-1}\ran C$ is closed. From the well known fact that a densely defined operator is closable if and only if its adjoint is densely defined we conclude that $(LC)^{-1}\ran C = \calH$. Making use of \eqref{e:applied_twice}, we obtain $UR\ker U = \{0\}$, which is (iii).
	
	(iii)$\Sra$(i). Let us first prove that $\{f_n : n\in I\}$ is linearly independent. We prove this for $I = \N$. A similar reasoning applies to the case $I = \Z$. Towards a contradiction, assume that there are $c_1,\ldots,c_N\in\C$, not all equal to zero, such that $\sum_{n=0}^Nc_nf_n = 0$. Without loss of generality, we may assume that $c_N\neq 0$. By (iii),
	\begin{equation}\label{e:toll}
	\sum_{n=0}^Nc_nf_{n+k} = 0\qquad\text{for all $k\in\N$}.
	\end{equation}
	Then $f_N\in\linspan\{f_0,\ldots,f_{N-1}\}$. But also, \eqref{e:toll} for $k=1$ implies that
	$$
	f_{N+1}\in\linspan\{f_1,\ldots,f_N\} = \linspan\{f_0,\ldots,f_{N-1}\}.
	$$
	Inductively, we conclude that $f_n\in\linspan\{f_0,\ldots,f_{N-1}\}$ for all $n\in\N$, which clearly is a contradiction as $\calH$ is infinite-dimensional. Hence, $F$ is linearly independent. The claim now follows from \cite{ch3,chr}.
	
	To prove \eqref{e:formula}, assume that the frame has the form $F=(T^nf_0)_{n\in \N}$ with a bounded operator $T$. Denote the frame operator of $F$ by $S$, i.e., $S = UC$. Let $f\in\calH$ and put $c := CS^{-1}f = (\<S^{-1}f,f_n\>)_{n\in I}\in\ell^2(I)$. Then $f = SS^{-1}f = \sum_{n\in I}c_nf_n$ and hence
	\begin{align*}
	\ol Tf = \sum_{n\in I}c_nf_{n+1} = \sum_{n\in I}(Rc)_nf_n = \sum_{i\in I}\left(RCS^{-1}f\right)_nf_n = URCS^{-1}f,
	\end{align*}
	which proves \eqref{e:formula}, since $U^\dagger = CS^{-1}$.
\end{proof}

\begin{rem}
Note that the question whether an overcomplete linearly independent frame is generated by a bounded operator highly depends on the order of its elements. This leads to the more general problem of finding necessary and sufficient conditions on a given linearly independent frame under which there exists a permutation of its elements resulting in a sequence that can be represented by a bounded operator. We will not address this  problem here.
\end{rem}

%

\section{\texorpdfstring{Orbits of the form $(T^nf_0)_{n\in \N}$}{Orbits indexed by the naturals}}\label{71219b}
In this section we consider frames $(T^nf_0)_{n\in\N}$ with a linear operator $T$ that is bounded and defined on the entire Hilbert space, i.e., $T\in L(\calH)$. We begin with a lemma concerning the similarity of operator-vector pairs.

\begin{lem}\label{l:similar-1}
Let $\calH_1$ and $\calH_2$ be Hilbert spaces and $(T_j,f_j)\in L(\calH_j)\times\calH_j$, $j=1,2$, such that $(T_1,f_1)\cong (T_2,f_2)$. Then   $(T_1^nf_1)_{n\in\N}$ is a frame for $\calH_1$ if and only if $(T_2^nf_2)_{n\in\N}$ is a frame for $\calH_2$. In the affirmative case, the operator $V$ in \eqref{e:similar} is unique.
\end{lem}
\begin{proof}
 If $V\in GL(\calH_1,\calH_2)$ is an operator as in \eqref{e:similar}, for $g\in\calH_2$ we have
$$
\sum_{n=0}^\infty|\<g,T_2^nf_2\>|^2 = \sum_{n=0}^\infty|\<g,VT_1^nV^{-1}Vf_1\>|^2 = \sum_{n=0}^\infty|\<V^*g,T_1^nf_1\>|^2.
$$
This shows that $(T_1^nf_1)_{n\in\N}$ is a frame for $\calH_1$ if and only if $(T_2^nf_2)_{n\in\N}$ is a frame for $\calH_2$. Moreover, for $f\in\calH_1$, $f = \sum_{n=0}^\infty c_nT_1^nf_1$, where $c\in\ell^2(\N)$, we have
$$
Vf = \sum_{n=0}^\infty c_nVT_1^nf_1 = \sum_{n=0}^\infty c_nT_2^nVf_1 = \sum_{n=0}^\infty c_nT_2^nf_2.
$$
Thus, $V$ is unique.
\end{proof}

In what follows, it is our aim to characterize the operator-vector pairs $(T,f_0)\in L(\calH)\times\calH$ for which the orbit $(T^nf_0)_{n\in\N}$ is a frame for $\calH$. It was already proved in \cite{cmpp} that the orbit $(T^nf_0)_{n\in\N}$ is a Riesz basis for $\calH$ if and only if
$$
(T,f_0)\cong (M_{\T}^+,1_\T).
$$
We will now show that this statement can be extended to obtain a characterization for all frame orbits $(T^nf_0)_{n\in\N}$ generated by a bounded operator $T$. To this end, we first state the following lemma. As Theorem \ref{t:charac_N} below will reveal, the frames appearing in it are -- up to similarity -- exactly the frames that we aim to characterize. Recall the notations $\calH_h$, $A_h$ and $\phi_h$ from \eqref{e:objects}.

\begin{lem}\label{l:prototype}
Let $h\in\frakI^*\cup\{0\}$. Then
\begin{equation}\label{e:yeah}
A_h^n\phi_h = P_{\calH_h}(z^n),\qquad n\in\N.
\end{equation}
In particular, if $h\neq 0$, $(A_h^n\phi_h)_{n\in\N}$ is an overcomplete Parseval frame for the infinite-dimensional Hilbert space $\calH_h$.
\end{lem}
\begin{proof}
Clearly, \eqref{e:yeah} holds for $h=0$. Thus, let $h\in\frakI^*$. Then  $\calH_h$ is always infinite-dimensional by Lemma \ref{l:rr}. The subspace $hL^2_+$ is invariant under $M_\T^+$ by Theorem \ref{t:beurling}. Since $L^2_{+}=hL^2_{+}\oplus \calH_h$, for $f\in L^2_+$ we have
\[
P_{\calH_h} M_\T^+f =P_{\calH_h} M_\T^+P_{\calH_h}f +  P_{\calH_h} M_\T^+P_{hL^2_+}f = P_{\calH_h} M_\T^+P_{\calH_h}f.
\]
Thus $P_{\calH_h} M_\T^+P_{\calH_h}=P_{\calH_h} M_\T^+,$ which implies that
\begin{equation}\label{0604a}(P_{\calH_h} M_\T^+P_{\calH_h})^n =P_{\calH_h} (M_\T^+)^n.
 \end{equation}
Therefore, by definition of $A_h$ and $\phi_h$, we get
\[
A_h^n\phi_h = P_{\calH_h} (M_\T^+)^nP_{\calH_h}1_{\T} = P_{\calH_h} (M_\T^+)^n 1_{\T}=P_{\calH_h} (z^n),
\]
which is \eqref{e:yeah}.
\end{proof}

We will frequently use the following fact; the definition of $C_0$-contractions is given in Appendix \rmref{ap:contractions}.

\begin{lem}[{\cite[Ch.\ III, Prop.\ 4.3]{nfbk}}]\label{l:C0}
Each operator $A_h$, $h\in\frakI$, is a $C_0$-contraction  and its minimal function is precisely $h$.
\end{lem}

We are now ready to prove our main result in this section.

\begin{thm}\label{t:charac_N}
Let $T\in L(\calH)$ and $f_0\in\calH$. Then the following statements are equivalent:
\begin{enumerate}
\item[{\rm (i)}]  The system $(T^nf_0)_{n\in\N}$ is a frame for $\calH$.
\item[{\rm (ii)}] $(T,f_0)\cong (A_h,\phi_h)$ for some $h\in\frakI^*\cup\{0\}$.
\end{enumerate}
In case these equivalent conditions are satisfied, the function $h$ in {\rm (ii)} is unique and the frame $(T^nf_0)_{n\in\N}$ is a Riesz basis for $\calH$ if and only if $h=0$.
\end{thm}
\begin{proof}
Due to Lemma \ref{l:prototype} and Lemma \ref{l:similar-1}, (ii) implies (i). Also, $(A_0^n\phi_0)_{n\in\N}$ is a Riesz basis for $\calH_0 = L^2_+$ and thus $(T^nf_0)_{n\in\N}$ is a Riesz basis for $\calH$ if $h=0$.

So, assume that (i) holds. By $U$ denote the $L^2_+$-synthesis operator of the frame $(T^nf_0)_{n\in\N}$, i.e., $U : L^2_+\to\calH$,
$$
U\vphi = \sum_{n=0}^\infty\<\vphi,z^n\>_{L^2_+}T^nf_0,\quad\vphi\in L^2_+.
$$
Since $(T^nf_0)_{n\in\N}$ is a frame for $\calH$, the operator $U$ is bounded and surjective. For $\vphi\in L^2_+$ we have $\<z\vphi,1_\T\>_{L^2_+}=0$ and thus
$$
TU\vphi = \sum_{n=0}^\infty\<\vphi,z^n\>_{L^2_+}T^{n+1}f_0 = \sum_{n=1}^\infty\<z\vphi,z^n\>_{L^2_+}T^nf_0 = U(z\vphi) = UM_\T^+\vphi.
$$
Therefore, $TU = UM_\T^+$. In particular, $\ker U\subset L^2_+$ is invariant under $M_\T^+$. By Theorem \ref{t:beurling} there exists a function $h\in\frakI\cup\{0\}$ such that $\ker U = hL^2_+$. Thus, $\ran U^* = L^2_+\ominus hL^2_+ = \calH_h$. Define the operator $V := U|_{\calH_h}\in L(\calH_h,\calH)$ and note that $V$ is boundedly invertible with $V^{-1} = U^*(UU^*)^{-1}$. In particular, $\calH_h$ is infinite-dimensional. This and Lemma \ref{l:rr} imply that $h\in\frakI^*\cup\{0\}$. Applying the operator $V^{-1}$ from the right to the equation $TU = UM_\T^+$ yields
$$
T = UM_\T^+V^{-1} = UP_{\ran U^*}M_\T^+V^{-1} = VA_hV^{-1}.
$$
And since also $V\phi_h = UP_{\ran U^*}1_\T = U1_\T = f_0$, (ii) follows. If $(T^nf_0)_{n\in\N}$ is a Riesz basis for $\calH$, then $\ker U = \{0\}$ and hence $h=0$.

The fact that $h$ in (ii) is unique follows from Lemma \ref{l:C0} and Lemma \ref{l:minimals}.
\end{proof}

\begin{rem}\label{r:para}
Theorem \ref{t:charac_N} provides a parametrization of the pairs $(T,f_0)$ generating a frame for $\calH$. Indeed, for each $h\in\frakI^*\cup\{0\}$ and every $V\in GL(\calH_h,\calH)$ the pair $(VA_hV^{-1},V\phi_h)$ generates a frame for $\calH$ and each frame $(T^nf_0)_{n\in\N}$ for $\calH$ with $T\in L(\calH)$ and $f_0\in\calH$ is of this form with unique $V$ and $h$ (cf.\ Lemma \ref{l:similar-1}). Hence, the parametrization is injective.
\end{rem}

As a consequence of Theorem \ref{t:charac_N} we now prove that if
$(T^nf_0)_{n\in\N}$ is an overcomplete frame and $T$ is bounded, then $T^nf\to 0$
as $n\to  \infty$ for all $f\in \calH.$ As already discussed in Section \ref{71218c}
this implies that the classical frames in $\ltr,$ e.g., Gabor frames, wavelet frames,
and frames of translates, only have representations of the form
$(T^nf_0)_{n\in\N}$ with a bounded operator $T$ if they form a Riesz basis.

\begin{cor}\label{c:strst}
Let $T\in L(\calH)$ and $f_0\in\calH$ such that $(T^nf_0)_{n\in\N}$ is an overcomplete frame for $\calH$. Then for each $f\in\calH$ we have that
\begin{equation}\label{e:00}
T^nf\to 0
\qquad\text{and}\qquad
(T^*)^{n}f\to 0\quad\text{as }n\to\infty.
\end{equation}
In particular, $(T^nf_0)_{n\in\N}$ does not contain any Riesz sequence.
\end{cor}
\begin{proof}
It was already proved in \cite{ap} that $ (T^*)^{n} f\to 0$ as $n\to\infty$ for all $f\in\calH$ (independent of whether the frame is overcomplete or not). By Theorem \ref{t:charac_N}, there exist an inner function $h\in\frakI^*$ and $V\in GL(\calH_h,\calH)$ such that $T = VA_hV^{-1}$. Therefore, it suffices to show that $A_h^nf\to 0$ for all $f\in\calH_h$. And indeed, for $f\in\calH_h$, via \eqref{0604a} and \eqref{e:projection}, we have
\begin{align*}
\left\|A_h^nf\right\|_{L^2}^2 &= \left\|P_{\calH_h}(M_\T^+)^n f\right\|_{L^2}^2 = \left\|P_{\calH_h} z^nf\right\|_{L^2}^2 = \left\|P_-(z^nf\ol h)\right\|_{L^2}^2 \\&= \sum_{k=1}^\infty\left|\left\<z^nf\ol h,z^{-k}\right\>_{L^2}\right|^2 = \sum_{k=n+1}^\infty\left|\left\<f\ol h,z^{-k}\right\>_{L^2}\right|^2,
\end{align*}
which tends to zero as $n\to\infty$.
\end{proof}

\begin{rem}
Corollary \ref{c:strst} can alternatively be proved by using results about $C_0$-contractions from \cite{nfbk} (see Appendix \ref{ap:contractions}). Indeed, as each operator $A_h$, $h\neq 0$, is a $C_0$-contraction, \eqref{e:00} follows directly from \cite[Ch.\ III, Prop.\ 4.2]{nfbk}.
\end{rem}

Let $T\in L(\calH)$ be an operator for which there exists some $f_0\in\calH$ such that $(T^nf_0)_{n\in\N}$ is a frame for $\calH$. An obvious natural question to ask is whether there exist other vectors $f\in\calH$ for which $(T^nf)_{n\in\N}$ also is a frame for $\calH$ and which vectors these are. Therefore, for arbitrary $T\in L(\calH)$ we define the set
$$
\calV_\N(T) := \left\{f\in\calH : (T^nf)_{n\in\N}\text{ is a frame for }\calH\right\}.
$$
The next proposition shows that from one vector $f_0\in\calV_\N(T)$ (if it exists) we obtain all vectors in $\calV_\N(T)$ by applying all invertible operators from the commutant $\{T\}'$ of $T$ to $f_0$. Recall that $\{T\}'$ is the set of all $V\in L(\calH)$ commuting with $T$.

\begin{prop}\label{p:VT}
Let $T\in L(\calH)$ and let $f_0\in\calV_\N(T)$. Then
\begin{equation}\label{e:VTN}
\calV_\N(T) = \left\{Vf_0 : V\in\{T\}'\cap GL(\calH)\right\}.
\end{equation}
\end{prop}
\begin{proof}
%
Clearly, if $V\in \{T\}'\cap GL(\calH)$, then $(T^nVf_0)_{n\in\N} = (VT^nf_0)_{n\in\N}$ is a frame for $\calH$. Let $f\in\calV_\N(T)$. By Theorem \ref{t:charac_N} there exist $g,h\in\frakI^*\cup\{0\}$ and $V\in GL(\calH_h,\calH)$, $W\in GL(\calH_g,\calH)$ such that
$$
T = VA_hV^{-1} = WA_gW^{-1}\qquad\text{and}\qquad f_0 = V\phi_h,\;f = W\phi_g.
$$
Hence, by Lemma \ref{l:C0} and Lemma \ref{l:minimals} we have $g=h$ (up to unimodular constant multiples). In particular, $\calH_h = \calH_g$, $A_h = A_g$, and $\phi_h = \phi_g$. We conclude that $f = W\phi_g = W\phi_h = WV^{-1}f_0$ and $WV^{-1}T = WA_hV^{-1} = WA_gV^{-1} = TWV^{-1}$.
\end{proof}

\begin{rem}\label{r:explicit}
A more explicit  characterization
of $\calV_\N(T)$ than given in \eqref{e:VTN} can be obtained as follows.
Let $M_\vphi$ denote the operator of multiplication with $\vphi$ in $L^2$.
If $h\in\frakI^*$, then $\{A_h\}' = \{P_{\calH_h}M_\vphi|_{\calH_h} : \vphi\in H^\infty\}$ by Sarason's Theorem (see, e.g., \cite[Lec.\ VIII.1]{n}). By the spectral mapping theorem (cf.\ \cite[Lec.\ III.3]{n}), the operator $P_{\calH_h}M_\vphi|_{\calH_h} = \vphi(A_h)$ is in $GL(\calH_h)$ if and only if the function $\vphi$ satisfies $\inf\{|h(z)|+|\vphi(z)| : z\in\D\} > 0$. Let us denote the class of functions $\vphi\in H^\infty$ satisfying this requirement by $\Theta_h$.

Now, let $(T^nf_0)_{n\in\N}$ form an overcomplete frame for $\calH$ and let $(T,f_0)$ be similar to $(A_h,\phi_h)$ via $W\in GL(\calH_h,\calH)$. Then $V\in\{T\}'\cap GL(\calH)$ is equivalent to $W^{-1}VW\in\{A_h\}'\cap GL(\calH_h)$. Hence, $V = W\vphi(A_h)W^{-1}$ for some $\vphi\in\Theta_h$. Therefore, $Vf_0 = W\vphi(A_h)\phi_h = WP_{\calH_h}(\vphi\phi_h)$, that is,
$$
\calV_\N(T) = \left\{WP_{\calH_h}(\vphi\phi_h) : \vphi\in\Theta_h\right\}.
$$
But due to \eqref{e:projection} we have $P_{\calH_h}(\vphi\phi_h) = hP_-(\vphi\ol h\phi_h) = hP_-(\vphi P_-\ol h) =hP_-(\vphi\ol h) = P_{\calH_h}\vphi$, and we obtain
$$
\calV_\N(T) = \left\{WP_{\calH_h}\vphi : \vphi\in\Theta_h\right\}.
$$
\end{rem}

\section{\texorpdfstring{Orbits of the form $(T^nf_0)_{n\in \Z}$}{Orbits indexed by the integers}}\label{71219c}
In this section we focus on bi-infinite orbits indexed by $\Z.$ The following lemma is proved similarly as Lemma \ref{l:similar-1}.

\begin{lem}\label{l:similar-2}
Let $\calH_1$ and $\calH_2$ be Hilbert spaces and $(T_j,f_j)\in GL(\calH_j)\times\calH_j$, $j=1,2$, such that $(T_1,f_1)\cong (T_2,f_2)$. Then $(T_1^nf_1)_{n\in\Z}$ is a frame for $\calH_1$ if and only if $(T_2^nf_2)_{n\in\Z}$ is a frame for $\calH_2$. In this case, the operator $V$ in \eqref{e:similar} is unique.
\end{lem}

We will now characterize the pairs $(T,f_0)\in GL(\calH)\times\calH$ which yield a frame for $\calH$ of the form $(T^nf_0)_{n\in\Z}$. As it will turn out below, up to similarity such frames are exactly the ones appearing in the following example.

\begin{ex}\label{ex:prototype}
It is clear that $(M_\T^n1_\T)_{n\in\Z} = (z^n)_{n\in\Z}$ is an orthonormal basis for $L^2(\T)$. Let $\sigma\in\frakB(\T)$, $|\sigma| > 0$. Then $(M_\sigma^n1_\sigma)_{n\in\Z}$ is a Parseval frame for $L^2(\sigma)$, since it is the orthogonal projection of the Riesz basis $ (M_\T^n1_\T)_{n\in\Z}$ for $L^2(\T)$ onto the closed subspace $L^2(\sigma)$ of $L^2(\T)$.
\end{ex}

Let us again start with the Riesz basis case. The next proposition shows that if $(T,f_0)$ generates a Riesz basis indexed by $\Z$, then $T$ is similar to the multiplication operator $M_\T$ on $L^2(\T)$.

\begin{prop}\label{p:charac_riesz}
Let $T\in GL(\calH)$ and $f_0\in\calH$. Then the following statements are equivalent:
\begin{enumerate}
\item[{\rm (i)}]  The system $(T^nf_0)_{n\in\Z}$ is a Riesz basis for $\calH$.
\item[{\rm (ii)}] $(T,f_0)\cong (M_{\T},1_\T)$.
\end{enumerate}
\end{prop}
\begin{proof}
%
The implication (ii)$\Sra$(i) is evident. For the converse direction assume that $(T^nf_0)_{n\in\Z}$ is a Riesz basis for $\calH$. Let $V\in GL(\calH)$ be such that $(VT^nf_0)_{n\in\Z}$ is an orthonormal basis (ONB) of $\calH$. Thus, $(S^ne_0)_{n\in\Z}$ is an ONB of $\calH$, where $S := VTV^{-1}$ and $e_0 := Vf_0$. Moreover, $(T,f_0)\cong (S,e_0)$. Put $e_n := S^ne_0$, $n\in\Z$. Then $(e_n)_{n\in\Z}$ is an ONB of $\calH$ and $Se_n = e_{n+1}$ for $n\in\Z$. Now, define the unitary map $U : \calH\to L^2(\T)$ by $Ue_n := z^n$, $n\in\Z$. It is easily seen that $USU^{-1} = M_\T$ and $Ue_0 = 1_\T$. Therefore,
$$
(T,f_0)\,\cong\,(S,e_0)\,\cong\,(USU^{-1},Ue_0) = (M_\T,1_\T),
$$
which completes the proof.
\end{proof}

\begin{lem}\label{l:necessary}
Let $T\in GL(\calH)$ and $f_0\in\calH$ such that $(T^nf_0)_{n\in\Z}$ is a frame for $\calH$ with frame operator $S$. Then $S^{-1/2}TS^{1/2}$ is unitary. In particular, $T$ is similar to a unitary operator.
\end{lem}
\begin{proof}
We have $Sf = \sum_{n\in\Z}\<f,T^nf_0\>T^nf_0$ for $f\in\calH$ and hence
$$
TST^*f = \sum_{n\in\Z}\<T^*f,T^nf_0\>T^{n+1}f_0 = \sum_{n\in\Z}\<f,T^{n+1}f_0\>T^{n+1}f_0 = Sf.
$$
Therefore, $TST^* = S$, meaning that $UU^* = \Id$, where $U := S^{-1/2}TS^{1/2}$. As $T\in GL(\calH)$, $U$ is a unitary operator.
\end{proof}

\begin{cor}\label{c:notendstozero}
Let $T\in GL(\calH)$ and $f_0\in\calH$ such that $(T^nf_0)_{n\in\Z}$ is a frame for $\calH$ with frame bounds $A$ and $B$. Then for all $f\in\calH$ and all $n\in\Z$ we have that
$$
\|T^nf\|\,\ge\,\sqrt{\frac A B}\,\|f\|\qquad\text{and}\qquad\| (T^*)^{n}f\|\,\ge\,\sqrt{\frac A B}\,\|f\|.
$$
\end{cor}
\begin{proof}
Define the operator $U := S^{-1/2}TS^{1/2}$, where $S$ denotes the frame operator of $(T^nf_0)_{n\in\Z}$. As $U$ is unitary, for $f\in\calH$ we have
$$
\left\|T^nf\right\| = \left\|S^{1/2}U^nS^{-1/2}f\right\|\ge\sqrt{A}\,\left\|U^nS^{-1/2}f\right\| = \sqrt{A}\,\left\|S^{-1/2}f\right\|\ge\sqrt{\frac A B}\,\|f\|.
$$
A similar calculation applies to $\|(T^*)^{n}f\|$.
\end{proof}

\begin{rem}
Lemma \ref{l:necessary} and Corollary \ref{c:notendstozero} can easily be generalized to systems of the form $(T^nf_j)_{n\in\Z,\,j\in J}$ with $T\in GL(\calH)$, $f_j\in\calH$, and $J$ being a countable index set.
\end{rem}

Consider two sets $\sigma_1,\sigma_2\in\frakB(\T)$ and their symmetric difference
$$
\sigma_1\triangle\sigma_2 := (\sigma_1\backslash\sigma_2)\cup(\sigma_2\backslash\sigma_1).
$$
Clearly, if $|\sigma_1\triangle\sigma_2|=0$, then $L^2(\sigma_1) = L^2(\sigma_2)$ and also $M_{\sigma_1} = M_{\sigma_2}$. In what follows, we shall identify sets $\sigma_1,\sigma_2\in\frakB(\T)$ whose symmetric difference has arc length measure zero. We also write $\sigma_1\subset\sigma_2$ if $|\sigma_1\backslash\sigma_2|=0$.

\begin{lem}\label{l:sigmas}
Let $\sigma_1,\sigma_2\in\frakB(\T)$. If $M_{\sigma_1}$ and $M_{\sigma_2}$ are similar, then $\sigma_1 = \sigma_2$.
\end{lem}
\begin{proof}
By $P_j$ we shall denote the orthogonal projection in $L^2(\T)$ onto $L^2(\sigma_j)$, $j=1,2$. Let $V\in GL(L^2(\sigma_1),L^2(\sigma_2))$ be such that $M_{\sigma_2} = VM_{\sigma_1}V^{-1}$. Then $VM_{\sigma_1}^n1_{\sigma_1} = M_{\sigma_2}^n\vphi$ for each $n\in\Z$, where $\vphi := V1_{\sigma_1}\in L^2(\sigma_2)$. Hence, $Vz_1^n = z_2^n\vphi$, where $z_j := P_jz\in L^2(\sigma_j)$, $j=1,2$. Let $g := \chi_{\sigma_1\setminus\sigma_2}\in L^2(\sigma_1)\subset L^2(\T)$ have the Fourier expansion $g = \sum_{n\in\Z}\alpha_nz^n$ with $(\alpha_n)_{n\in\Z}\in\ell^2(\Z)$. Then we have
$$
Vg = VP_1g = V\sum_{n\in\Z}\alpha_nz_1^n = \sum_{n\in\Z}\alpha_nz_2^n\vphi = \left(\sum_{n\in\Z}\alpha_nz_2^n\right)\vphi = (P_2g)\vphi = 0.
$$
Since $V$ is injective, we conclude that $\chi_{\sigma_1\setminus\sigma_2}=g=0$ and thus $|\sigma_1\backslash\sigma_2|=0$. By interchanging the roles of $\sigma_1$ and $\sigma_2$ we also obtain $|\sigma_2\backslash\sigma_1|=0$.
\end{proof}

The next theorem is the main result in this section.

\begin{thm}\label{t:charac}
Let $T\in GL(\calH)$ and $f_0\in\calH$. Then the following statements are equivalent:
\begin{enumerate}
\item[{\rm (i)}]  The system $(T^nf_0)_{n\in\Z}$ is a frame for $\calH$.
\item[{\rm (ii)}] $(T,f_0)\cong (M_{\sigma},1_\sigma)$ for some $\sigma\in\frakB(\T)$, $|\sigma|>0$.
\end{enumerate}
In case these equivalent conditions are satisfied, the set $\sigma\in\frakB(\T)$ is unique and the frame $(T^nf_0)_{n\in\Z}$ is a Riesz basis for $\calH$ if and only if $\sigma = \T$.
\end{thm}
\begin{proof}
The implication (ii)$\Sra$(i) is clear, since $(M_\sigma^n1_\sigma)_{n\in\Z}$ is a frame for $L^2(\sigma)$, cf.\ Example \ref{ex:prototype}. For the other direction, assume that $(T^nf_0)_{n\in\Z}$ is a frame for $\calH$. Then the operator $U : L^2\to\calH$, defined by
$$
U\vphi := \sum_{n\in\Z}\<\vphi,z^n\>_{L^2}T^nf_0,\qquad\vphi\in L^2,
$$
is well defined, bounded, and surjective. For $\vphi\in L^2$ we have
$$
TU\vphi = \sum_{n\in\Z}\<\vphi,z^n\>T^{n+1}f_0 = \sum_{n\in\Z}\<\vphi,z^{n-1}\>T^nf_0 = \sum_{n\in\Z}\<z\vphi,z^n\>T^nf_0 = U(z\vphi) = UM_\T\vphi
$$
and therefore $TU = UM_\T$ (and, equivalently, $T^{-1}U = UM_\T^*$). In particular, $\ker U$ is both $M_\T$- and $M_\T^*$-invariant. That is, $M_\T\ker U = \ker U$. Theorem \ref{t:beurling} now implies that $\ker U = L^2(\sigma_0)$ with some $\sigma_0\in\frakB(\T)$. Hence, $\ran U^* = L^2\ominus \ker U = L^2(\sigma)$, where $\sigma = \T\backslash\sigma_0$. The operator $V := U|_{L^2(\sigma)}$ thus maps $L^2(\sigma)$ bijectively onto $\calH$. Applying its inverse from the right to the equation $TU = UM_\T$ yields $T = UM_\T V^{-1} = VM_\sigma V^{-1}$. Also, $V1_\sigma = UP_{L^2(\sigma)}1 = U1 = f_0$. Hence, (ii) holds. The remaining claims in Theorem \ref{t:charac} follow from Lemma \ref{l:sigmas} and Proposition \ref{p:charac_riesz}.
\end{proof}

\begin{rem}\label{r:Z}
If $\sigma\in\frakB(\T)$, the spectrum $\sigma(M_\sigma)$ of $M_\sigma$ does not necessarily coincide with $\sigma$ in the usual sense as $\sigma$ might be non-closed. In fact, it can be easily proved that $\sigma(M_\sigma)$ is the essential closure $\ol\sigma^e$ of $\sigma$ (cf.\ \cite{gmz}). Due to \cite[Lemma 2.11]{gmz} we always have $|\sigma\backslash\ol\sigma^e|=0$, but $|\ol\sigma^e\backslash\sigma|$ might be positive (see, e.g., \cite[Example 2.12]{gmz}) so that $|\sigma\triangle\ol\sigma^e| > 0$. Hence, following our convention, we always have $\sigma\subset\sigma(M_\sigma)$ but not necessarily $\sigma = \sigma(M_\sigma)$. The latter holds, e.g., if $\sigma$ is closed.
\end{rem}

\begin{ex}
Let $T_1$ denote the operator of translation in $L^2(\R)$ by $1$, i.e., $(T_1f)(x) = f(x+1)$, $f\in L^2(\R)$. It is well known that $(T_1^nf_0)_{n\in\Z}$ is a frame for its closed linear span $\calH\subset L^2(\R)$ if and only if the periodic function $\Phi(\omega) := \sum_{n\in\Z}|\wh f_0(\omega+n)|^2$ is bounded above and below by positive constants on $[0,1]\backslash N$, where $N = \{\omega : \Phi(\omega) = 0\}$. It is our aim to calculate the set $\sigma$ in Theorem \ref{t:charac} in this case.

So, let $(T_1^nf_0)_{n\in\Z}$ be a frame for $\calH$. By Theorem \ref{t:charac} there exist $\sigma\in\frakB(\T)$ and $V\in GL(L^2(\sigma),\calH)$ such that $T_1^nf_0 = VM_\sigma^n1_\sigma$, $n\in\N$. Henceforth, we shall transform everything from the unit circle to $[0,1]$ so that, e.g., $\sigma\subset [0,1]$ and $z^n$ becomes $e^{2\pi in\cdot}$. Applying the Fourier transform $\calF$ to $T_1^nf_0 = VM_\sigma^n1_\sigma$ and denoting $W := \calF V$, we obtain $e^{2\pi in\cdot}\wh f_0 = W(e^{2\pi in\cdot}1_\sigma)$, $n\in\N$. Hence, for $h\in L^2(0,1)$, $h=\sum_{n\in\Z}c_ne^{2\pi in\cdot}$, we have $\wt h\wh f_0 = W(h1_\sigma)$, where $\wt h$ denotes the periodic extension of $h$ to $\R$. Thus,
$$
\|h1_\sigma\|_{L^2(\sigma)}^2 \sim \|W(h1_\sigma)\|_2^2 = \int_\R|\wt h\wh f_0|^2\,d\omega = \int_0^1|h(\omega)|^2\Phi(\omega)\,d\omega.
$$
This shows that, indeed, $\Phi=0$ on $[0,1]\backslash\sigma$ and that $\Phi$ is bounded above and below a.e.\ on $\sigma$, that is, $\sigma = [0,1]\backslash N$.
\end{ex}

Similarly as in Section \ref{71219b}, for $T\in GL(\calH)$ we define the set
$$
\calV_\Z(T) := \left\{f\in\calH : (T^nf)_{n\in\Z}\text{ is a frame for }\calH\right\}.
$$
The next two statements are analogues of Proposition \ref{p:VT} and Remark \ref{r:explicit} for the case of orbits indexed by $\Z$.

\begin{prop}\label{p:VTZ}
Let $T\in GL(\calH)$ and let $f_0\in\calV_\Z(T)$. Then
$$
\calV_\Z(T) = \left\{Vf_0 : V\in\{T\}'\cap GL(\calH)\right\}.
$$
\end{prop}
\begin{proof}
If $V\in \{T\}'\cap GL(\calH)$, then $(T^nVf_0)_{n\in\N} = (VT^nf_0)_{n\in\N}$ is a frame for $\calH$. Let $f\in\calV_\Z(T)$. By Theorem \ref{t:charac} there exist $\sigma_1,\sigma_2\in\frakB(\T)$, $|\sigma_1|,|\sigma_2|>0$, as well as $V\in GL(L^2(\sigma_1),\calH)$, $W\in GL(L^2(\sigma_2),\calH)$ such that
$$
T = VM_{\sigma_1}V^{-1} = WM_{\sigma_2}W^{-1}\qquad\text{and}\qquad f_0 = V1_{\sigma_1},\;f = W1_{\sigma_2}.
$$
Hence, by Lemma \ref{l:sigmas} we have $\sigma_1=\sigma_2$ (in the sense of Remark \ref{r:Z}(b)). Therefore, $f = W1_{\sigma_2} = W1_{\sigma_1} = WV^{-1}f_0$ and $WV^{-1}T = WM_{\sigma_1}V^{-1} = WM_{\sigma_2}V^{-1} = TWV^{-1}$.
\end{proof}

\begin{rem}
Let $\sigma\in\frakB(\T)$, $|\sigma|>0$. It is easy to see that $\{M_\sigma\}'$ consists of the operators of multiplication with functions $\psi\in L^\infty(\sigma)$. Such an operator is in $GL(L^2(\sigma))$ if and only if $\operatorname{essinf}|\psi| > 0$. Therefore, if $T\in GL(\calH)$ and $f_0\in\calH$ such that $(T^nf_0)_{n\in\Z}$ is a frame for $\calH$ and $(T,f_0)$ is similar to $(M_\sigma,1_\sigma)$ via $W\in GL(L^2(\sigma),\calH)$, it follows that
$$
\calV_\Z(T) = \left\{W\psi : \psi\in L^\infty(\sigma),\,\operatorname{essinf}|\psi| > 0\right\}.
$$
\end{rem}

\section{Construction of frame orbits} \label{80104a}
For two non-negative sequences $(a_j)_{j\in\N}$ and $(b_j)_{j\in\N}$ we write $a_j\sim b_j$ if there exist constants $c,C > 0$ such that $c a_j\le b_j\le C a_j$ for all $j\in\N$. Recall that a sequence $\Lambda = (\la_j)_{j\in\N}$ in the open unit disk $\D$ is called {\em uniformly separated} if
\begin{eqnarray} \label{81701a}
\delta_\Lambda := \inf_{j\in\N}\prod_{k\neq j}\left|\frac{\la_j-\la_k}{1 - \ol{\la_j}\la_k}\right| > 0.
\end{eqnarray} The condition \eqref{81701a} is known in the literature under the name
{\it the Carleson condition.}
The frame characterization in the next theorem was already proved in \cite{acmt,ap} (see also \cite{cmpp} for a more general version). In addition, we here also provide estimates for the frame bounds.

\begin{thm}\label{t:acmt}
Let $T\in L(\calH)$ be a normal operator and $f_0\in\calH$. Then $(T^nf_0)_{n\in\N}$ is a frame for $\calH$ if and only if $T = \sum_{j=0}^\infty\la_j\<\,\cdot\,e_j\>e_j$, where $(\la_j)_{j\in\N}\subset\D$ is uniformly separated, $(e_j)_{j\in\N}$ is an orthonormal basis for $\calH$, and $|\<f_0,e_j\>|^2\sim 1-|\la_j|^2$. In this case, the frame $(T^nf_0)_{n\in\N}$ has the frame bounds $\alpha\Delta^{-1}$ and $\beta\Delta$, where
\begin{equation}\label{e:alphaDelta}
\alpha := \inf_{j\in\N}\frac{|\<f_0,e_j\>|^2}{1-|\la_j|^2},
\qquad
\beta := \sup_{j\in\N}\frac{|\<f_0,e_j\>|^2}{1-|\la_j|^2},
\qquad\text{and}\qquad
\Delta := \frac 2{\delta_\Lambda^4}\left(1-2\log\delta_\Lambda\right).
\end{equation}
\end{thm}
\begin{proof}
Recall that $H^2 := H^2(\D)$ is the Hardy space on the unit disk $\D$ consisting of functions of the form $\vphi(z) = \sum_{n=0}^\infty c_nz^n$, $z\in\D$, where $(c_n)_{n\in\N}\in\ell^2(\N)$. Then $\|\vphi\|_{H^2} = \|c\|_{\ell^2}$. Let us also define the Fourier transforms $\mathbb F : \ell^2(\N)\to\calH$ and $\calF : \ell^2(\N)\to H^2$ by $\mathbb Fc := \sum_{j\in\N}c_je_j$ and $\calF c(z) := \sum_{n\in\N}c_nz^n$. By $U$ we denote the synthesis operator of the frame $(T^nf_0)_{n\in\N}$. Then for $c\in\ell^2(\N)$ we have
\begin{align*}
Uc = \sum_{n=0}^\infty c_n\sum_{j=0}^\infty\la_j^n\<f_0,e_j\>e_j = \sum_{j=0}^\infty\<f_0,e_j\>\left(\sum_{n=0}^\infty c_n\la_j^n\right)e_j = \sum_{j=0}^\infty\<f_0,e_j\>(\calF c)(\la_j)e_j.
\end{align*}
Now, define the interpolation operator $T_\Lambda : H^2\to\ell^2(\N)$ by
$$
T_\Lambda\vphi := \left((1-|\la_j|^2)^{1/2}\vphi(\la_j)\right)_{j\in\N},\quad\vphi\in H^2.
$$
As $\Lambda$ is uniformly separated, the operator $T_\Lambda$ is well-defined, bounded, and onto (see, e.g., \cite{ss}). Denoting $\hat c := \calF c$, we have
$$
Uc = \sum_{j=0}^\infty\frac{\<f_0,e_j\>}{\sqrt{1-|\la_j|^2}}(T_\Lambda\hat c)_je_j = \mathbb FDT_\Lambda\hat c,
$$
where $D = \operatorname{diag}_{j\in\N}(\<f_0,e_j\>(1-|\la_j|^2)^{-1/2})$. That is,
$$
U = \mathbb FDT_\Lambda\calF.
$$
Since $T_\Lambda$ is onto, its adjoint $T_\Lambda^*$ is bounded from below, i.e., we have $\|T_\Lambda^*c\|_{H^2}\ge\gamma'\|c\|_{\ell^2}$ for all $c\in\ell^2(\N)$ with some $\gamma' > 0$. The supremum $\gamma$ of all such $\gamma'$ is called the reduced minimum modulus of $T_\Lambda^*$ and can be written as follows in terms of $T_\Lambda$ (cf., e.g., \cite[Lemma 2.1]{dp}):
$$
\gamma = \inf\left\{\|T_\Lambda\vphi\|_{\ell^2} : \vphi\in H^2\ominus\ker(T_\Lambda),\,\|\vphi\|_{H^2}=1\right\}.
$$
In order to estimate $\gamma$, let $\vphi\in H^2\ominus\ker(T_\Lambda)$ with $\|\vphi\|_{H^2}=1$. By \cite[Lemma 3]{ss}, there exists some $\phi\in H^2$ such that $T_\Lambda\phi = T_\Lambda\vphi$ and $\|\phi\|_{H^2}^2\le\Delta\|T_\Lambda\phi\|_{\ell^2}^2$. As $\phi - \vphi\in\ker(T_\Lambda)$ and $\vphi\in\ker(T_\Lambda)^\perp$, we have $\|\phi\|_{H^2}\ge\|\vphi\|_{H^2} = 1$ and thus
$$
\|T_\Lambda\vphi\|_{\ell^2}^2\ge\Delta^{-1}\|\phi\|_{H^2}^2\ge\Delta^{-1},
$$
which yields $\gamma\ge\Delta^{-1/2}$. Finally, for $f\in\calH$ we obtain
$$
\|U^*f\|_{\ell^2}^2 = \|\calF^*T_\Lambda^*D^*\mathbb F^*f\|_{\ell^2}^2\,\ge\,\gamma^2\|D^*\mathbb F^*f\|_{\ell^2}^2\,\ge\,\Delta^{-1}\alpha\|f\|^2.
$$
This proves that $\alpha\Delta^{-1}$ is indeed a lower frame bound of $(T^nf_0)_{n\in\N}$. The optimal upper frame bound is $\|U\|^2$. Hence, an upper frame bound is given by $\|D\|^2\|T_\Lambda\|^2 = \beta\|T_\Lambda\|^2$. By \cite[Lemma 1 and Lemma 3]{ss} we have $\|T_\Lambda\|^2\le\Delta$.
\end{proof}

The next proposition generalizes Theorem \ref{t:acmt} to the  class of operators that are similar to a normal operator. As Theorem \ref{t:acmt}, it is easy to apply the result to obtain  explicit constructions. Note in particular the condition (ii), which clearly relates to the general characterization of frames
$(T^nf)_{n\in\N}$ in Theorem \ref{t:charac_N}; we will comment on that in Remark \ref{r:comp_normal}. Recall that if $(g_j)_{j\in\N}$ is a Riesz basis for $\calH$, then there exists a unique sequence $(h_j)_{j\in\N}$ such that $(g_j)_{j\in\N}$ and $(h_j)_{j\in\N}$ are bi-orthogonal, i.e., $\<g_j,h_k\> = \delta_{jk}$ for $j,k\in\N$. This sequence is then also a Riesz basis for $\calH$ and is called the {\em dual basis} of $(g_j)_{j\in\N}$. We denote it by $(g_j')_{j\in\N}$.

\begin{prop}\label{p:similar_normal}
For $T\in L(\calH)$ and $f_0\in\calH$ the following statements are equivalent.
\begin{enumerate}
\item[{\rm (i)}]   $T$ is similar to a normal operator and $(T^nf_0)_{n\in\N}$ is a frame for $\calH$.
\item[{\rm (ii)}]  $(T,f_0)\cong (A_h,\phi_h)$, where $h$ is a Blaschke product defined by a uniformly separated sequence.
\item[{\rm (iii)}] $T = \sum_{j=0}^\infty\la_j\<\,\cdot\,,g_j\>g_j'$, where $(\la_j)_{j\in\N}\subset\D$ is uniformly separated, $(g_j)_{j\in\N}$ is a Riesz basis for $\calH$, and $|\<f_0,g_j\>|\sim\sqrt{1-|\la_j|^2}$.
\end{enumerate}
In case {\rm (iii)} holds, the frame $(T^nf_0)_{n\in\N}$ has the frame bounds $\alpha\Delta^{-1}B^{-1}$ and $\beta\Delta A^{-1}$, where $\Delta$ is as in \eqref{e:alphaDelta},
$$
\alpha := \inf_{j\in\N}\frac{|\<f_0,g_j\>|^2}{1-|\la_j|^2},
\qquad
\beta := \sup_{j\in\N}\frac{|\<f_0,g_j\>|^2}{1-|\la_j|^2},
$$
and $A$ and $B$ are the Riesz bounds of $(g_j)_{j\in\N}$.
\end{prop}
\begin{proof}
It was shown in \cite{ka} (see also \cite[p.\ 212]{n}) that $A_h$ ($h\in\frakI^*$) is similar to a normal operator if and only if $h$ is a Blaschke product defined by a uniformly separated sequence. This proves the equivalence of (i) and (ii). Here, $h=0$ can be excluded since $A_0 = M_\T^+$ is clearly not similar to a normal operator. The equivalence of (i) and (iii) follows from Theorem \ref{t:acmt} and the simple fact that, fixing an orthonormal basis $(e_j)_{j\in\N}$ of $\calH$, two sequences $(g_j)_{j\in\N}$ and $(h_j)_{j\in\N}$ are bi-orthogonal Riesz bases for $\calH$ if and only if $g_j = (W^{-1})^*e_j$ and $h_j = We_j$ with some $W\in GL(\calH)$.

Assume now that (iii) holds, let $(e_j)_{j\in\N}$ be an orthonormal basis for $\calH$, and choose $W\in GL(\calH)$ such that $g_j = (W^{-1})^*e_j$ and $g_j' = We_j$, $j\in\N$. Then  $T = WNW^{-1}$, where $N = \sum_j\la_j\<\,\cdot\,,e_j\>e_j$. In particular, $(T^nf_0)_{n\in\N} = (WN^nW^{-1}f_0)_{n\in\N}$, which has frame bounds $a_0\inf\sigma(L)$ and $b_0\sup\sigma(L)$, where $L = WW^*$ and $a_0$ and $b_0$ denote frame bounds of $(N^nW^{-1}f_0)_{n\in\N}$. By Theorem \ref{t:acmt} we know that we can take $a_0 = \alpha\Delta^{-1}$ and $b_0 = \beta\Delta$. Now, if $f\in\calH$, $f = \sum_j c_jg_j$ with $c\in\ell^2(\N)$, then
$$
\<Lf,f\> = \|W^*f\|^2 = \Bigg\|\sum_{j=0}^\infty c_je_j\Bigg\|^2 = \sum_{j=0}^\infty|c_j|^2,
$$
and hence $\<Lf,f\>\ge B^{-1}\|f\|^2$ and $\<Lf,f\>\le A^{-1}\|f\|^2$, which implies $\inf\sigma(L)\ge B^{-1}$ and $\sup\sigma(L)\le A^{-1}$.
%
\end{proof}

\begin{rem}\label{r:comp_normal}
Consider $T\in L(\calH)$ and $f_0\in\calH$ such that $(T^nf_0)_{n\in\N}$ is a frame for $\calH$. Then by Theorem \ref{t:charac_N}, $T$ is similar to $A_h$ for some unique $h\in\frakI^*\cup\{0\}$. Our results show that
\begin{itemize}
	\item $(T^nf_0)_{n\in\N}$ is a Riesz basis if and only if $h=0$.
	\item $T$ is similar to a normal operator if and only if $h$ is a Blaschke product generated by a uniformly separated sequence.
\end{itemize}
Since the Blaschke products generated by uniformly separated sequences form a very small set within the class of all inner functions, it follows that the  frames $(T^nf_0)_{n\in\N}$ generated by an operator $T\in L(\calH)$ that is similar to a normal operator are very particular. In other words - the class of frames $(T^nf_0)_{n\in\N}$ that are known so far (i.e, the Riesz bases and the constructions arising from Theorem \ref{t:acmt}) form a very small subclass of all possible constructions. It remains a very interesting and challenging open problem to give a more concrete and constructive way of obtaining such frames for general bounded operators $T$.
\end{rem}

In the following theorem we construct frames of the form $(T^nf_0)_{n\in\N}$ with a non-normal operator $T\in L(\calH)$ by means of perturbations of frames as considered in Theorem \ref{t:acmt}.

\begin{thm} \label{80104b}
Let $(\la_j)_{j\in\N}$ be a uniformly separated sequence in $\D$, $(e_j)_{j\in\N}$ an orthonormal basis for $\calH$, and $f_0\in\calH$ such that $|\<f_0,e_j\>|\sim\sqrt{1 - |\la_j|^2}$. For fixed $k,\ell\in\N$, $k\neq\ell$, and $\tau\in\C$, $\tau\neq(\la_k-\la_\ell)\tfrac{\<f_0,e_\ell\>}{\<f_0,e_k\>}$, consider the operator $T_\tau\in L(\calH)$, defined by
$$
T_\tau f := \sum_{j=0}^\infty\la_j\<f,e_j\>e_j + \tau\<f,e_k\>e_\ell,\quad f\in\calH.
$$
Then $(T_\tau ^nf_0)_{n\in\N}$ is a frame for $\calH$. For $\tau\neq 0$ the operator $T_\tau$ is not normal.
\end{thm}
\begin{proof}
Setting $d := \la_k-\la_\ell$, we define sequences $(g_j)_{j\in\N}$ and $(h_j)_{j\in\N}$ by $g_j := h_j := e_j$ for $j\notin\{\ell,k\}$ and
$$
g_\ell := e_\ell - \ol\tau\ol d^{-1} e_k,\quad
g_k := \ol d^{-1}e_k,\quad
h_\ell := e_\ell,\quad
h_k := \tau e_\ell + de_k.
$$
Note that $d\neq 0$ as $(\la_j)_{j\in\N}$ is uniformly separated. A simple calculation shows that $(g_j)_{j\in\N}$ is a Riesz basis for $\calH$ with $g_j' = h_j$, $j\in\N$, and that $T_\tau = \sum_{j=0}^\infty\la_j\<\,\cdot\,,g_j\>g_j'$. In order to apply Proposition \ref{p:similar_normal} it remains to show that $|\<f_0,g_j\>|\sim\sqrt{1-|\la_j|^2}$. Since we already know that $|\<f_0,e_j\>|\sim\sqrt{1-|\la_j|^2}$, it suffices that $\<f_0,g_\ell\>\neq 0$ and $\<f_0,g_k\>\neq 0$. But this is exactly the case when $\tau\neq(\la_k-\la_\ell)\tfrac{\<f_0,e_\ell\>}{\<f_0,e_k\>}$. The fact that $T_\tau$ is not normal for $\tau\neq 0$ simply follows from $\|T_\tau e_k\|^2 = |\la_k|^2 + |\tau|^2$ and $\|T_\tau^*e_k\|^2 = |\la_k|^2$.
\end{proof}

\appendix

\vspace*{.2cm}
\section{Contractions}\label{ap:contractions}
By $H^\infty$ denote the set of all bounded analytic functions on the open unit disk $\D$. It is well known that each $h\in H^\infty$ has radial limits on the unit circle $\T$ almost everywhere, that is,
$$
h^*(z) := \lim_{r\upto 1}h(rz)
$$
exists for a.e. $z\in\T$. The function $h^*$ is then an element of $L^\infty = L^\infty(\T)$ with vanishing negative Fourier coefficients, i.e., $h^*\in L^2_+$ (cf.\ \eqref{e:L2+}). Conversely, for every such function $g\in L^\infty$ there exists some $h\in H^\infty$ such that $g = h^*$. A function $h\in H^\infty$ with $|h^*(z)| = 1$ for a.e.\ $z\in\T$ is called an {\em inner function}.

A contraction on a Hilbert space $\calH$ is an operator $T\in L(\calH)$ with operator norm $\|T\|\le 1$.  A contraction $T$ is said to be {\em completely non-unitary} (c.n.u.), if there is no non-trivial subspace $M\subset\calH$ reducing $T$ such that $T|M$ is unitary. For a c.n.u.\ contraction $T$ there exists an $H^\infty$-functional calculus, which is defined as follows:
$$
u(T)f := \lim_{r\upto 1}u_r(T)f,\qquad f\in\calH,\;u\in H^\infty,
$$
where $u_r(z) := u(rz)$, $r\in (0,1)$, $z\in\D$, see \cite[Ch.\ III.2]{nfbk}. The operator $u_r(T)$ is well defined because if $u(z) = \sum_{n=0}^\infty\alpha_nz^n$, then $u_r(z) = \sum_{n=0}^\infty r^n\alpha_nz^n$, $z\in\D$, so that $(r^n\alpha_n)_{n\in\N}\in\ell^1(\N)$ for every $r\in (0,1)$. The mapping $u\mapsto u(T)$ is an algebra homomorphism from $H^\infty$ to $L(\calH)$.

A c.n.u.\ contraction $T$ is said to belong to the class $C_0$ if there exists some non-trivial $u\in H^\infty$ which annihilates $T$, that is, $u(T) = 0$. Among the non-vanishing $T$-annihilating $H^\infty$-functions there exists a {\em minimal function} $m_T$  which is a divisor of all of them. The minimal function is always an inner function and can be seen as a generalization of the minimal polynomial of linear operators on finite-dimensional spaces. In particular, it determines the spectrum of $T$ uniquely (cf.\ \cite[Ch.\ III, Thm.\ 5.1]{nfbk}). We shall make use of the following simple lemma, which directly follows from \cite[Ch.\ III, Prop.\ 4.6]{nfbk}.

\begin{lem}[{\cite{nfbk}}]\label{l:minimals}
If $S$ and $T$ are $C_0$-contractions which are similar to each other, then $m_S = m_T$ \braces{up to unimodular constant multiples}.
\end{lem}

\section{\texorpdfstring{Subspaces of $L^2(\T)$}{Subspaces of L\texttwosuperior(T)}}\label{ap:subspaces}
The following theorem is due to Beurling and Helson (see, e.g., \cite[Ch.\ I.1]{n}).

\begin{thm}[\cite{n}]\label{t:beurling}
{\rm (a)} Let $\calL\subset L^2$ be a closed $M_\T$-invariant subspace. If $M_\T\calL = \calL$, then there exists $\sigma\in\frakB(\T)$ such that $\calL = L^2(\sigma)$. If $M_\T\calL\neq\calL$, then $\calL = hL^2$ with some $h\in\frakI$.

\smallskip
{\rm (b)} Let $\calL\subset L^2_+$ be a closed $M_\T^+$-invariant subspace. Then $\calL = hL^2_+$ with some $h\in\frakI\cup\{0\}$.
\end{thm}

\begin{lem}\label{l:anders}
Let $h\in\frakI$. Then
\begin{equation}\label{e:oc}
P_{hL^2_-}L^2_+ = hL^2_-\cap L^2_+ = L^2_+\ominus hL^2_+.
\end{equation}
Moreover,
\begin{equation}\label{e:projection}
P_{hL^2_-\cap L^2_+}f = h\cdot P_-(f\ol h),\qquad f\in L^2_+.
\end{equation}
\end{lem}
\begin{proof}
First of all, we have $(hL^2_+)^\perp = hL^2_-$, since
$$
f\in (hL_+^2)^\perp\;\Llra\;\forall g\in L^2_+\!: \,\<f\ol h,g\>_{L^2} = 0\;\Llra\;f\ol h\in L^2_-\;\Llra\;f\in hL^2_-.
$$
In particular, $(hL^2_+)^\perp\cap L^2_+ = hL^2_-\cap L^2_+$. It is now easily checked that $(z^{-n}h)_{n=1}^\infty$ is an orthonormal basis for $hL^2_-$. Hence, for $f\in L^2$,
\begin{equation}\label{e:gen_projection}
P_{hL^2_-}f = \sum_{n=1}^\infty\<f,hz^{-n}\>hz^{-n} = h\sum_{n=1}^\infty\<f\ol h,z^{-n}\>z^{-n} = h\cdot P_-(f\ol h).
\end{equation}
Now, if $m\in\N$, we have
\begin{align*}
P_{hL^2_-}(z^m)
&= h\cdot P_-(z^m\ol h) = h\cdot P_-\left(z^m\sum_{n=0}^\infty\ol{\<h,z^n\>}z^{-n}\right) = h\cdot P_-\left(\sum_{n=-m}\ol{\<h,z^{n+m}\>}z^{-n}\right)\\
&= h\cdot\sum_{n=1}^\infty\ol{\<h,z^{n+m}\>}z^{-n} = z^mh\sum_{n=m+1}^\infty\ol{\<h,z^n\>}z^{-n} = z^mh\ol{\sum_{n=m+1}^\infty\<h,z^n\>z^{n}}\\
&= z^mh\ol{\left(h - \sum_{n=0}^m\<h,z^n\>z^n\right)} = z^m - h\cdot \sum_{n=0}^m\ol{\<h,z^n\>}z^{m-n}\\
&= z^m - \sum_{n=0}^m\ol{\<h,z^{m-n}\>}z^nh,
\end{align*}
which proves that $P_{hL^2_-}(z^m)\in L^2_+$. By continuity of $P_{hL^2_-}$ it follows that $P_{hL^2_-}L^2_+\subset hL^2_-\cap L^2_+$. For the converse inclusion, let $f\in hL^2_-\cap L^2_+$. Then $f = P_{hL^2_-}f\in P_{hL^2_-}L^2_+$. This proves \eqref{e:oc}. Finally, \eqref{e:projection} follows from \eqref{e:oc} and \eqref{e:gen_projection}.
\end{proof}

The next result can be found in \cite{rr}.

\begin{lem}[{\cite[Theorem 3.14]{rr}}]\label{l:rr}
Let $h\in\frakI$. Then $L^2_+\ominus hL^2_+$ is finite-dimensional if and only if $h$ is a finite Blaschke product.
\end{lem}

\vspace{1cm}
\section*{Author affiliations}

\end{document}